\documentclass{amsart} \usepackage[normalem]{ulem}
\usepackage{enumitem} \usepackage{xcolor} \usepackage{amscd}
\usepackage{amssymb, amsmath}

\usepackage{tikz}
\usetikzlibrary{shapes.multipart, positioning,
decorations.pathreplacing, decorations.markings,
decorations.pathmorphing, calc, matrix, cd, intersections}

\usepackage[all, knot]{xy}
\usepackage{hyperref}
 \usepackage[calc]{picture}
\usepackage{graphicx}

\theoremstyle{plain} 

\newtheorem*{introthm*}{Theorem}

\newtheorem{cor}[equation]{Corollary}
\newtheorem{lem}[equation]{Lemma}
\newtheorem{lemma}[equation]{Lemma}
\newtheorem{prop}[equation]{Proposition}
\newtheorem*{prop*}{Proposition}
\newtheorem{theorem}[equation]{Theorem}

\theoremstyle{definition}
\newtheorem{defn}[equation]{Definition}
\newtheorem{chunk}[equation]{}

\newlength{\thmtopspace}                
\newlength{\thmbotspace}                
\newlength{\thmheadspace}               
\newlength{\thmindent}                  
\renewcommand{\subparagraph}{\vspace*{\thmbotspace}}

\theoremstyle{remark}
\newtheorem{rem}[equation]{Remark}

\makeatletter
\newsavebox\myboxA
\newsavebox\myboxB
\newlength\mylenA
\newcommand*\xoverline[2][0.75]{%
    \sbox{\myboxA}{$\m@th#2$}%
    \setbox\myboxB\null
    \ht\myboxB=\ht\myboxA%
    \dp\myboxB=\dp\myboxA%
    \wd\myboxB=#1\wd\myboxA
    \sbox\myboxB{$\m@th\overline{\copy\myboxB}$}
    \setlength\mylenA{\the\wd\myboxA}
    \addtolength\mylenA{-\the\wd\myboxB}%
    \ifdim\wd\myboxB<\wd\myboxA%
       \rlap{\hskip 0.5\mylenA\usebox\myboxB}{\usebox\myboxA}%
    \else
        \hskip -0.5\mylenA\rlap{\usebox\myboxA}{\hskip 0.5\mylenA\usebox\myboxB}%
    \fi}

\newcommand*\xunderline[2][0.75]{%
    \sbox{\myboxA}{$\m@th#2$}%
    \setbox\myboxB\null
    \ht\myboxB=\ht\myboxA%
    \dp\myboxB=\dp\myboxA%
    \wd\myboxB=#1\wd\myboxA
    \sbox\myboxB{$\m@th\underline{\copy\myboxB}$}
    \setlength\mylenA{\the\wd\myboxA}
    \addtolength\mylenA{-\the\wd\myboxB}%
    \ifdim\wd\myboxB<\wd\myboxA%
       \rlap{\hskip 0.5\mylenA\usebox\myboxB}{\usebox\myboxA}%
    \else
        \hskip -0.5\mylenA\rlap{\usebox\myboxA}{\hskip 0.5\mylenA\usebox\myboxB}%
    \fi}
\makeatother

\renewcommand{\o}[1]{\xoverline{#1}}

\makeatletter
\newcommand*\bigcdot{\mathpalette\bigcdot@{.5}}
\newcommand*\bigcdot@[2]{\mathbin{\vcenter{\hbox{\scalebox{#2}{$\m@th#1\bullet$}}}}}
\makeatother

\def\op{\text{op}}

\newcommand{\cA}{\mathcal{A}}
\newcommand{\cB}{\mathcal{B}}
\newcommand{\cC}{\mathcal{C}}
\newcommand{\cD}{\mathcal{D}}

\newcommand{\cS}{\mathcal{S}}

\newcommand{\cU}{\mathcal{U}}

\newcommand{\bN}{\mathbb{N}}

\newcommand{\bZ}{\mathbb{Z}}

\renewcommand{\t}[1]{\widetilde{#1}}

\newcommand{\la}{\langle}
\newcommand{\ra}{\rangle}

\newcommand{\xra}[1]{\xrightarrow{#1}}

\newcommand{\Hom}[3]{\operatorname{Hom}_{#1}(#2,#3)}

\newcommand{\h}[1]{\mathcal{H}^{#1}}

\numberwithin{equation}{section}

\renewcommand{\cD}{\mathcal{D}}
\newcommand{\Dl}[1][0]{\cD^{\leq #1}} 
\newcommand{\Dg}[1][0]{\cD^{\geq #1}} 
\newcommand{\dtstr}[1][\cD]{(#1^{\leq 0}, #1^{\geq 0})}
\newcommand{\Dcminus}{\cD^{-}_{c}} \newcommand{\Dcb}{\cD^{b}_{c}}

\newcommand{\Coprod}{\operatorname{Coprod}}
\newcommand{\Add}{\operatorname{Add}}
\newcommand{\smd}{\operatorname{smd}}
\newcommand{\D}[1]{\mathsf{D}( #1 )}

\newcommand{\Dp}[1]{\mathsf{D}_{\mathsf{perf}}( #1 )}
\newcommand{\Dqc}[1]{\mathsf{D}_{\mathsf{qc}}( #1 )}
\newcommand{\Dbc}[1]{\mathsf{D}^{\mathsf{b}}_{\mathsf{coh}}( #1 )}
\newcommand{\Dmc}[1]{\mathsf{D}^{-}_{\mathsf{coh}}( #1 )}

\begin{document}
\title{Gluing Approximable Triangulated Categories}
\author{Jesse Burke}
\author{Amnon Neeman}
\author{Bregje Pauwels}
\address{Centre for Mathematics and its Applications, Mathematical Sciences Institute, John
  Dedman Building, The Australian National University, Canberra, ACT 2601, AUSTRALIA}
\email{jesse.burke@anu.edu.au}
\email{amnon.neeman@anu.edu.au}
\email{bregje.pauwels@anu.edu.au}
\begin{abstract}
Given a bounded-above cochain complex of modules over a ring, it is standard to replace it by a projective resolution, and it is classical that doing so can be very useful. 

Recently, a modified version of this was introduced in triangulated categories other than the derived category of a ring. A triangulated category is \emph{approximable} if this modified procedure is possible. Not surprisingly this has proved a powerful tool. For example: the fact that $\Dqc X$ is approximable when X is a quasi compact, separated scheme led to major improvements on old theorems due to Bondal, Van den Bergh and Rouquier. 
In this article we prove that, under weak hypotheses, the recollement of two approximable triangulated categories is approximable. In particular, this shows many of the triangulated categories that arise in noncommutative algebraic geometry are approximable.
\end{abstract}
\maketitle

\setcounter{section}{-1}
\section{Introduction}
Let $\cD$ be a triangulated category  with a t-structure $\dtstr$. One can consider two objects $x,\ y\in\cD$ as ``close together'' if there exists an exact triangle
$x\to y\to z$
in $\cD$ with $z\in\Dl[-n]$ for some large $n$. This intuitive
definition of distance was used in \cite{ApproxAmnon} to define two
new notions, assuming $\cD$ has coproducts and is compactly generated by a single object $G$.
The first is the property of (weak) 
approximability: $\cD$ is \emph{approximable} if every object in $\Dl$ 
can be approximated by objects that are finitely built out of arbitrary coproducts of certain shifts of $G$.
That is, every object in $\Dl$ has a sequence of ``simpler'' objects converging to it.
The second is a pair of subcategories $\Dcb \subseteq \Dcminus
\subseteq \cD$.
Here, $\Dcminus$ denotes the full subcategory with objects that can be approximated by compact objects.

The two notions are connected: if $\cD$ is
approximable and $G$ satisfies certain finiteness conditions, then every finite
cohomological functor, respectively locally finite cohomological
functor, on the compact objects $\cD^{c}$ of $\cD$ is represented by
an object of $\Dcb,$ respectively $\Dcminus$, by \cite[Theorem
0.3]{ApproxAmnon}. In this paper, we prove that approximability and
the two distinguished subcategories may be
glued along a recollement. To be precise, we let $\cD_{F}$ and $\cD_{U}$ be triangulated categories with a single compact generator and let $\cD$ be a compactly generated triangulated category. 
We assume
there is a recollement
\begin{displaymath}
\begin{tikzcd}
\cD_{F} \arrow[r, "i_{*}"] & \cD \arrow[l,shift left = 4.3,
"i^{!}"']  \arrow[l,shift right = 4.3, "i^{*}"'] \arrow[r,
"j^{*}"] & \cD_{U},  \arrow[l,shift left = 4.3, "j_{*}"']
\arrow[l, shift right = 4.3, "j_{!}"']
\end{tikzcd}
\end{displaymath}
which implies that $\cD$ has a compact generator $G$.

\begin{prop*}
  If $\cD_U$ is weakly approximable, $\Hom {\cD}
  {\Sigma^{-n} G} {G} = 0$ for $n \gg 0$,
  and $j^{*} G$ is in
  $(\cD_{U})^{-}_{c},$ then there  are equalities:
  \begin{align*}
  	\Dcminus &= \left \{ X \in \cD \, | \, i^{*} X \in  (\cD_{F})^{-}_{c}
  	\text{ and } j^{*} X \in  (\cD_{U})^{-}_{c} \right \},\\
  	\Dcb \,
  	&= \left \{ X \in \cD \, | \, i^{*} X \in
  	(\cD_{F})^{-}_{c},\  j^{*} X \in (\cD_{U})^{b}_{c} \text{ and }  i^{!} X \in
  	(\cD_{F})^{+}\right \}.
  \end{align*}
\end{prop*}
\begin{introthm*}
  If $\cD_{F}$ and $\cD_{U}$ are approximable and 
  $\Hom {\cD}{\Sigma^{-n} G} {G} = 0$ for~$n \gg 0,$ then $\cD$ is approximable.
\end{introthm*}

Approximability is a very useful and natural notion in the study of the triangulated categories of algebraic geometry. If $\cD = \Dqc X$ is the derived
category of quasi-coherent sheaves on a scheme $X,$ then $\cD$ is
approximable if $X$ is quasi-compact and separated \cite[Example 3.6]{ApproxAmnon},
and if we further assume $X$ Noetherian, the subcategories $\Dcb$ and
$\Dcminus$ are $\Dbc{X}$ and $\Dmc{X}$, respectively.
The approximability property of the category $\Dqc X$ is used in
\cite{AmnonStrong} and \cite{ApproxAmnon} 
to significantly generalize foundational results of Bondal and
Van den Bergh \cite{MR1996800}, and Rouquier \cite{MR2434186}; in
particular to mixed characteristic.

We expect that approximability will also be useful in noncommutative geometry. We mean noncommutative algebraic geometry in the sense
of \cite{MR3545926}, where a \emph{noncommutative scheme} is a small pretriangulated dg-category with a
classical generator, and the derived category of the scheme is the
derived category of the dg-category. 
In \cite[Theorem 3.1.1(ii)]{MR1996800} it was proved that if $X$ is a quasi-compact quasi-separated scheme, then $\Dp{X}$ has a classical generator. In this way, the classical commutative schemes are considered within the world of noncommutative schemes.

A central construction in noncommutative algebraic geometry is the 
gluing $\cA$ of two small dg-categories $\cB, \cC$ along a $\cC$-$\cB$-bimodule (\cite {MR3545926, MR2451292}).
In the final section of this paper, we show that if
$\cB$ and $\cC$ have approximable derived categories, and the bimodule is cohomologically 
bounded above, then the glued dg-category $\cA$ also has an approximable derived
category. Furthermore, we describe how $\D{\cA}_c^{-}$ is built from gluing $\D{\cB}_c^{-}$ and $\D{\cC}_c^{-}$.
This result significantly expands the known
examples of approximable triangulated categories in noncommutative
algebraic geometry.  Even if $\D{\cB}$ and $\D{\cC}$ are ``commutative'',
i.e., equivalent to the derived categories of quasi-coherent sheaves
on a commutative scheme, the glued category $\D{\cA}$ is rarely
commutative.

Many interesting classes of noncommutative schemes are
constructed via gluing. In \cite{MR3439086}, Kuznetsov and Lunts construct categorical resolutions of separated schemes of finite type over a field $k$ of characteristic zero by gluing derived categories of smooth varieties. This notion of categorical resolutions was defined in \cite{MR2609187} as a generalization of Van Den Bergh's notion of noncommutative crepant resolution \cite{MR2077594}. Gluing of DG-categories and categorical resolutions are also used in~\cite{efimov} to show that $\Dbc{X}$ is homotopically finitely presented whenever $X$ is a separated scheme of finite type over a field $k$ of characteristic zero. 
For other places where gluing is used in noncommutative geometry, see \cite{MR3535370, MR3728631, MR2981713, efimov1}.


\section{Notation and definitions}
Throughout, $\cD$ denotes a triangulated category with coproducts.
We recall some standard notation, see~\cite[Definition 0.21]{ApproxAmnon}.
\begin{chunk}
  Let $\cA,\ \cB$ be full subcategories of $\cD$, and define the
  following full subcategories:
  \begin{enumerate}[label=\alph*)]
    \item
    $\cA\ast\cB$ has for objects all the $x\in\cD$ such that there exists an exact triangle $a\to x\to b$ in $\cD$ with $a\in\cA$ and $b\in\cB$.
    \item $\Add(\cA)$ has for objects all coproducts of objects of $\cA$.
    \item $\smd(\cA)$ has for objects all direct summands of objects of $\cA$.
  \end{enumerate}
\end{chunk}

\begin{chunk}
  Let $G$ be an object in $\cD$ and suppose $n\in\bN$ and
  $A\leq B\in\bZ\cup\{-\infty,+\infty\}$.  We define the following
  full subcategories:
  \begin{enumerate}[label=\alph*)]
    \item If $A,B\in\bZ$, then $G[A,B]\subset\cD$ has objects $\{\Sigma^{-i} G\mid A\leq i\leq B\}$. \\
    Similarly,
    $G(-\infty,B]\subset\cD$ has objects $\{\Sigma^{-i} G\mid  i\leq B\}$.
    We define $G[A,\infty)$ and $G(-\infty,\infty)$ analogously.
    \item $\Coprod_n(G[A,B])$ is defined inductively on the integer $n$ by setting 
	$$\Coprod_1(G[A,B])=\Add(G[A,B])\ \ \text{and}$$ 
	$$\Coprod_{n+1}(G[A,B])=\Coprod_1(G[A,B])\ast\Coprod_{n}(G[A,B]).$$
    \item $\Coprod(G[A,B])$ is the smallest full subcategory $\cS\subset \cD$, closed under coproducts, with $\cS\ast\cS\subset\cS$ and with $G[A,B]\subset\cS$.
    \item $\o{\la G \ra}_n^{[A,B]}=\smd(\Coprod_{n}(G[A,B]))$.
    
    \item $\o{\la G \ra}^{[A,B]}=\smd(\Coprod(G[A,B]))$.
  \end{enumerate}
\end{chunk}

\begin{chunk} \cite[Definition 0.21]{ApproxAmnon}\label{def:wa} A
  triangulated category $\cD$ with coproducts is \emph{weakly approximable} if there
  exists a compact generator $G,$ a t-structure $\dtstr,$ and an
  integer $A > 0,$ such that the following hold:
  \begin{enumerate}
    \item $\Sigma^{A} G \in \Dl$ and $\Hom {\cD} {\Sigma^{-A} G} \Dl = 0.$
  
    \item For every $X \in \Dl,$ there exists an exact triangle $E \to X \to D$ with $E \in \o{\la G \ra}^{[-A,A]}$ and $D \in \Dl[-1].$
  \end{enumerate}
  If properties $(1)$ and $(2)$ hold for a compact generator $G$ and
  integer $A,$ then for every compact generator $H$ there is an
  integer $A_{H}$
  for which they hold, by~\cite[Proposition
  2.6]{ApproxAmnon}.
  
  The category $\cD$ is \emph{approximable} if we can choose $A > 0$ such that, moreover,
  \begin{enumerate}\setcounter{enumi}{2}
  	\item in the exact triangle $E \to X \to D$ above, we may assume $E \in \o{\la G \ra}^{[-A,A]}_{A}.$
  \end{enumerate}
   If properties $(1)$, $(2)$ and $(3)$ hold for a compact generator $G$ and
  integer $A,$ then for every compact generator $H$ there is an
  integer $A_{H}$
  for which they hold, by~\cite[Proposition
  2.6]{ApproxAmnon}.
\end{chunk}

\begin{chunk}\label{def:rec} A
  \emph{recollement} is a diagram of exact functors between
  triangulated categories,
  \begin{displaymath}
    \begin{tikzcd}
      \cD_{F} \arrow[r, "i_{*}"] 
      & \cD \arrow[l,shift left = 4.3, "i^{!}"']  \arrow[l,shift right
      = 4.3, "i^{*}"'] \arrow[r, "j^{*}"] & \cD_{U}, \arrow[l,shift
      left = 4.3, "j_{*}"'] \arrow[l, shift right = 4.3, "j_{!}"']
    \end{tikzcd}
  \end{displaymath}
  with the following properties:
  \begin{enumerate}
    \item each functor is left adjoint to the one below it;
    \item $j^{*} i_{*} = 0$;
          
    \item there exist natural transformations $d: i_{*}i^{*} X \to j_{!}
    j^{*}[1]$ and $d': j_{*}j^{*} \to i_{*}i^{!}[1]$ such that
    for any $X \in \cD,$ the following are exact
    triangles:
    \[j_{!}j^{*}X\xra{\eta} X\xra{\epsilon} i_{*}i^{*} X \xra{d} j_{!}j^{*}X[1]\]
    \[i_{*}i^{!}X\xra{\eta} X\xra{\epsilon} j_{*}j^{*} X
    \xra{d'} i_{*}i^{!}X[1],\]
    where $\eta$ and $\epsilon$ are the (co)unit maps of the adjunctions;
          
    \item $i_{*}, j_{!}, j_{*}$ are fully faithful.
  \end{enumerate} 
  We note that $i^*,\ i_{*},\ j_{!},\ j^*$ preserve coproducts, since they are left adjoints. Since $i^{*}$ and $j_{!}$ have right adjoints preserving coproducts,
  \cite[Theorem~5.1]{Ne96} informs us that  $i^{*}$ and $j_{!}$
  respect compact objects.
\end{chunk}
        
\begin{chunk}
  (\cite[1.4.10]{MR751966})\label{def:gluing} Consider a recollement
  as above.  Given t-structures $\dtstr [\cD_{F}]$ and
  $\dtstr [\cD_{U}]$ on $\cD_{F}$ and $\cD_{U}$ respectively, the
  \emph{glued t-structure on $\cD$} has aisle
	$$\Dl:=\{X\in\cD\mid i^{*}X\in \Dl_F \mbox{ and } j^{*}X\in\Dl_U \}$$
	and co-aisle
	$$\Dg:=\{X\in\cD\mid i^{!}X\in \Dg_F \mbox{ and } j^{*}X\in\Dg_U \}.$$
\end{chunk}

\begin{chunk}
	(\cite[1.3.16]{MR751966})\label{def:texact}
	Let $\cD_1$, $\cD_2$ be triangulated categories endowed with t-structures.
	A functor $f:\cD_1\to\cD_2$ is \emph{right t-exact} if $f(\Dl_1)\subset \Dl_2$, and \emph{left t-exact} if $f(\Dg_1)\subset \Dg_2$. We say $f$ is \emph{t-exact} if $f$ is left and right t-exact.

	Consider a recollement as in~\ref{def:rec} and suppose $\cD$ has a t-structure glued from t-structures on $\cD_{F}$ and $\cD_{U}$.
	By~\cite[1.3.17(iii)]{MR751966}, the functors $i^*,\ j_{!}$ are right t-exact, $i_*,\ j^*$ are t-exact and $i^!,\ j_*$ are left t-exact. 
\end{chunk}

      \section{$t$-structure generated by $G$}
      
\begin{defn} \cite[Definition 0.10]{ApproxAmnon}\label{def:equiv}
	Let $\cD$ be a triangulated category. Two t-structures $\dtstr[\cD_1]$ and $\dtstr[\cD_2]$ on $\cD$ are \emph{equivalent} if there exists an integer $A > 0$ with $\Dl[-A]_1\subset\Dl_2\subset\Dl[A]_1$.
\end{defn}

\begin{lem}\label{lem:stablegl}
	The gluing of t-structures along a recollement is stable under equivalence. That is, consider a recollement as in~\ref{def:rec} and suppose $\dtstr [\cD_{F}]$, $\dtstr [{\cD'}_{F}]$ are equivalent t-structures on $\cD_F$ and $\dtstr [\cD_{U}]$, $\dtstr [{\cD'}_{U}]$ are equivalent t-structures on $\cD_U$.
	Then the t-structure on $\cD$ obtained by gluing $\dtstr [\cD_{F}]$, $\dtstr [\cD_{U}]$ is equivalent to the t-structure obtained by gluing $\dtstr [{\cD'}_{F}]$, $\dtstr [{\cD'}_{U}]$.
\end{lem}

\begin{proof}
	Write $\dtstr$, $\dtstr [{\cD'}]$ for the t-structures on $\cD$ obtained by gluing $\dtstr [\cD_{F}]$, $\dtstr [\cD_{U}]$ and $\dtstr [{\cD'}_{F}]$, $\dtstr [{\cD'}_{U}]$ respectively.
	If $A>0$ is an integer such that $\Dl[-A]_{F}\subset{\cD'}_{F}^{\leq 0}\subset \Dl[A]_{F}$ and $\Dl[-A]_U\subset{\cD'}^{\leq 0}_U\subset \Dl[A]_U$, then
	$\Dl[-A]\subset{\cD'}^{\leq 0}\subset \Dl[A]$.
\end{proof}

      \begin{defn}\cite[Theorem A.1]{MR1974001}
        Given a compact object $G$ of $\cD$, the t-structure
        \emph{generated by $G$}, denoted $(\Dl_{G}, \Dg_{G})$, has
        aisle
        $\Dl_{G} =\o {\la G \ra}^{(-\infty, 0]}.$
      \end{defn}
  
 \begin{rem}
 		Since the category $\Coprod(G(-\infty,0])$ is closed under positive suspensions and coproducts, it contains all direct summands of its objects. This shows  $\Dl_{G} =\o {\la G \ra}^{(-\infty, 0]}=\Coprod(G(-\infty,0]))$.
 \end{rem}

Suppose now that $\cD$ is compactly generated by the compact object $G$. If $H$ is also a compact generator for $\cD$, \cite[Lemma 0.9]{ApproxAmnon} provides a positive integer $A$ with $H\in \o{\la G \ra}_A^{[-A,A]}$ and
$G\in \o{\la H \ra}_A^{[-A,A]}$. This shows
$\Dl[-A]_H\subset \Dl_G\subset \Dl[A]_H$,
hence the t-structures generated by $G$ and $H$ are equivalent. Thus
the next definition does not depend on the choice of compact generator:

\begin{defn}
	If $\cD$ is compactly generated by the compact object $G$, then the \emph{preferred equivalence class of t-structures} is the equivalence class containing the t-structure $\dtstr[\cD_G]$ generated by $G$.
\end{defn}

The first part of the following lemma follows from~\cite[Remark 0.20]{ApproxAmnon}, and the second part follows from~\cite[Observation 0.12(ii) and Lemma 2.8]{ApproxAmnon}:

\begin{lemma}\label{lem:gen}
	Suppose $\cD$ has a compact generator $G$ with
	$\Hom {\cD}{\Sigma^{-n} G} {G} = 0$ for~$n \gg 0$.
        Assume $\dtstr[\cD]$ is a t-structure in the preferred
        equivalence class. Then
	\begin{enumerate}
		\item there exists an integer $C > 0$ such that $\Hom {\cD} G {\Dl [-C]} = 0$;
		\item for all compact objects $F,H\in \cD,$ 
		we have $\Hom {\cD}{\Sigma^{-m} F} {H} = 0$ for $m \gg 0$.
	\end{enumerate}
	In particular, the vanishing hypothesis on $G$ is independent of the
	choice of compact generator.
\end{lemma}

Recall that a t-structure $\dtstr$ on $\cD$ is
      \emph{non-degenerate} if
      \[\bigcap_{n\in\bZ} \Dl[n]=0 \quad \text{and} \quad \bigcap_{n\in\bZ}\Dg[n]=0.\]

\begin{lemma}\label{lem:nondegnG}
  Suppose $\cD$ has a compact generator $G$ with
  $\Hom {\cD}{\Sigma^{-n} G} {G} = 0$ for~$n \gg 0$. Then every t-structure in the preferred equivalence class
  is non-degenerate.
\end{lemma}

\begin{proof}
Being non-degenerate is clearly stable under equivalence, so it suffices to show the t-structure $(\Dl_{G}, \Dg_{G})$ generated by $G$ is non-degenerate.
  Let $X$ be a nonzero object in $\cD$. Since $G$ is a compact
  generator for $\cD$, there exists a nonzero morphism
  $\Sigma^l G\to X$ for some $l\in\bZ$. It follows that
  $X\notin \Dg[-l+1]_G=(\Dl[-l]_G)^{\perp}$.  Moreover, by the first
  part of~Lemma~\ref{lem:gen} there exists $C>0$ such that
  $\Hom {\cD} {G} {\Dl[-C]_G} = 0$. Hence
  $\Hom {\cD} {\Sigma^{l} G} {\Dl[-C-l]_G} = 0$ and we can conclude
  that $X\notin \Dl[-C-l]_G$.
\end{proof}

The following lemma is exactly \cite[Proposition 2.6]{ApproxAmnon}.  It tells us that once we know a category to be (weakly) approximable, any compact generator and any t-structure in the preferred equivalence class will fulfill the approximability criteria in~Definition~\ref{def:wa}.

\begin{lem}\label{rem:wa}
  Let $\cD$ be weakly approximable, suppose $G$ is a
  compact generator for $\cD$ and that $\dtstr$ is any t-structure in the preferred equivalence class.  Then there exists $A>0$ such that conditions (1) and (2) in~Definition~\ref{def:wa} hold for
  the compact generator $G$ and the t-structure $\dtstr$.
  If $\cD$ is moreover approximable, $A$ can be chosen to further satisfy condition (3) in~Definition~\ref{def:wa}.
\end{lem}


\begin{lemma}\label{lem:jG}
  Let $\cD$ be a weakly approximable category, let $G$ be a 
  compact generator, and let $\dtstr[\cD]$ be a t-structure in the preferred
  equivalence class. There exists an integer $A>0$ such that,
  if $\Hom {\cD}{\Sigma^{-n}G} {X} = 0$ for all $n > 0$, then $X\in\Dl[A]$.
  Consequently if $\Hom {\cD}{\Sigma^{-n}G} {X} = 0$ for $n\gg 0$
  then $X\in \Dl[m]$ for some integer $m$.
\end{lemma}

\begin{proof}
By~Lemma~\ref{rem:wa} there exists an integer $A>0$, such that conditions (1) and (2) in~Definition~\ref{def:wa} hold for the compact generator $G$ and the t-structure $(\Dl, \Dg)$.
For any integer $m$ and for the object $X\in\cD$ in the statement of
the current lemma, we
consider the exact triangle
	$$\tau^{\leq m}X\to X\to \tau^{> m}X,$$
where $\tau^{\leq m}, \,\tau^{>m}$ are the truncations with respect to the t-structure $(\Dl, \Dg)$. Applying~\ref{def:wa} (2) to $\tau^{\leq m}X$ gives an exact triangle $E\to \tau^{\leq m}X\to D$ in~$\cD$,
	$$\begin{tikzcd}
          E
          \arrow{d}&&\\
          \tau^{\leq m}X \arrow{r} \arrow{d} &X \arrow{r}
          &\tau^{>m}X\\
          D \arrow[dashrightarrow]{ru}&&
	\end{tikzcd}$$ with
        $E \in \o{\la G \ra}^{[m-A,m+A]}$ and
$D \in \Dl [m-1]$.  Now suppose $m\geq A+1$.
Because 
$\Hom {\cD}{\Sigma^{-n} G} {X} = 0$ for all $n > 0$ and
$E\in\o{\la G \ra}^{[m-A,m+A]}\subset\o{\la G \ra}^{[1,m+A]}$,  we have that 
         there is no nonzero map from $E$ to~$X$. Hence
        $\tau^{\leq m}X\to X$ factors through $\tau^{\leq m}X\to D$.
        Using that $D \in \Dl[m-1]$ this shows
	$$\tau^{\leq m}X\cong \tau^{\leq m-1}X,$$
	in other words $\h m X=0$ for all $m> A$.

Since $\cD$ is
        weakly approximable we know $\Hom {\cD} {\Sigma^{-n} G} {G} = 0$ for $n \gg 0,$ hence by~Lemma~\ref{lem:nondegnG} the t-structure
        $(\Dl, \Dg)$ is non-degenerate.
        By
        \cite[1.3.7]{MR751966}, 
        $X$ is in~$\Dl[A]$. 
      \end{proof}

\begin{lem}\label{lem:neg-gen-means-tstrs-agree}
  Consider a recollement as in~Definition~\ref{def:rec}. Suppose
  $\cD_U$ is weakly approximable with compact generator $G_{U}$, and $\cD$ has a compact
  generator $G'$ such that $\Hom {\cD}
  {\Sigma^{-n} G'} {G'} = 0$ for $n \gg 0$.  We can then find a compact
  generator $G$ of $\cD$ such that the t-structure on $\cD$ generated by $G$ is equal to the gluing of the
  t-structure on $\cD_{F}$ generated by $i^{*} G$ and the t-structure on
  $\cD_{U}$ generated by~$G_{U}$.
\end{lem}

\begin{proof}
  Write $(\Dl_{U}, \Dg_{U})$ for the t-structure on $\cD_{U}$
  generated by~$G_{U}$. We note that $j_{!}G_U$ is compact in $\cD$
  because the functor $j_{!}$ preserves compactness (see~\ref{def:rec}).  The second part
  of~Lemma~\ref{lem:gen} now shows
  $$\Hom {\cD_{U}} {\Sigma^{-n} G_U} {j^{*} G'} \cong \Hom
  {\cD}{\Sigma^{-n} j_{!}G_U} {G'} = 0$$ for $n \gg 0$.  It follows
  that $j^{*} G'\in\Dl[m]_{U}$ for some $m>0$ by~Lemma~\ref{lem:jG}.
  Now let $$G:=\Sigma^m G'\oplus j_{!} G_U,$$
  which is clearly a
  compact generator for $\cD$. We write
  $(\Dl_{F}, \Dg_{F})$ for the t-structure on $\cD_{F}$ generated by
  $i^{*} G$ and $(\Dl_\text{gl}, \Dg_\text{gl})$ for the t-structure
  on $\cD$ obtained by gluing $(\Dl_{F}, \Dg_{F})$ and
  $(\Dl_{U}, \Dg_{U})$.  It remains to show that
	$$\Dl_G=\Dl_\text{gl},$$
	where $(\Dl_{G}, \Dg_{G})$ is the t-structure on $\cD$
        generated by $G$.  First, we note that $i^{*}G\in \Dl_{F}$ and
        $j^{*}G= \Sigma^m j^{*}G'\oplus G_U\in \Dl_{U}$.  Using that
        the functors $i^{*}$ and $j^{*}$ commute with coproducts, this
        shows $i^{*}\Dl_{G}\subset \Dl_{F}$ and
        $j^{*}\Dl_{G}\subset \Dl_{U}$.  Hence,
        $$\Dl_{G}\subset \Dl_\text{gl}.$$ On the other hand, we
        clearly have $j_{!}G_U\in \Dl_G$ and thus
        $j_{!}\Dl_{U}\subset\Dl_{G}$, because the functor $j_{!}$
        commutes with coproducts. Since $i_{*}i^{*}G$ fits in an exact triangle
	$$j_{!}j^{*}G\to G\to i_{*}i^{*}G$$
	with $j_{!}j^{*}G, \,G\in \Dl_{G}$, we see that
        $i_{*}i^{*}G\in\Dl_{G}$.  It follows that
        $i_{*}\Dl_{F}\subset \Dl_{G}$.  Now, let $X\in\Dl_\text{gl}$
        and consider the exact triangle
	$$j_{!}j^{*}X\to X\to i_{*}i^{*}X.$$
	By the above, $j_{!}j^{*}X,\, i_{*}i^{*}X\in\Dl_{G}$, so
        $X\in\Dl_{G}$.  We conclude that $\Dl_{G}= \Dl_\text{gl}$.
      \end{proof}
  
Combining the above lemma with~Lemma~\ref{lem:stablegl}, we get:
  
  \begin{cor}\label{importantcorollary}
  	Consider a recollement as in~Definition~\ref{def:rec}. Suppose
  	$\cD_U$ is weakly approximable and $\cD$ has a compact
  	generator $G$ such that $\Hom {\cD}
  	{\Sigma^{-n} G} {G} = 0$ for~$n \gg 0$. Then the t-structure on $\cD$ obtained by gluing t-structures on $\cD_F$ and $\cD_U$, both of which
        are in the preferred equivalence class, is in the preferred equivalence class.
  \end{cor}

\section{Gluing $\Dcminus$}
\begin{defn}\cite[Definition 0.16]{ApproxAmnon}\ \ 
  Let $\cD$ be a triangulated category with t-structure $\dtstr.$
  Recall that we write $\cD^+:=\bigcup_{m > 0} \Dg [-m]$,  $\cD^-:=\bigcup_{m > 0} \Dl [m]$ and $\cD^{b} = \cD^+\cap\cD^-$.
  The full subcategory $\Dcminus$ has for objects all the $X\in\cD$ such that,
  for any integer
  $m > 0$, there exists an exact triangle $E \to X \to E'$ with $E$ compact
  and $E' \in \Dl [-m]$. The subcategory $\Dcb$ is
  defined to be $\Dcminus \cap \cD^{b}$.
\end{defn}

\begin{rem}\label{preferred-Ds}
Observe that equivalent t-structures yield equal
$\cD^+,\ \cD^-,\ \cD^b,\ \Dcminus,\ \Dcb$. 
If $\cD$ has a single compact generator, form the subcategories
$\cD^+,\ \cD^-,\ \cD^b,\ \Dcminus,\ \Dcb$
corresponding to the preferred equivalence
class of t-structures. These are intrinsic. 
\end{rem}

\begin{rem}\label{thickness}
	Suppose $\cD$ has a compact generator $G$ such that $\Hom {\cD}{\Sigma^{-n} G} {G} = 0$ for $n \gg 0$, and endow $\cD$ with a t-structure in the preferred equivalence class. 
	By~\cite[Proposition~0.19 and Remark~0.20]{ApproxAmnon},
	the subcategories $\Dcb \subseteq \Dcminus$ are thick triangulated subcategories of~$\cD^{-}.$

        If $\cD$ is weakly approximable then the above applies: one
        easily sees that $\cD$ has a compact
        generator $G$ with $\Hom {\cD}{\Sigma^{-n} G} {G} = 0$ for $n\gg0$.
\end{rem}

\begin{lemma}
 Consider a recollement as in~Definition~\ref{def:rec}. Suppose
    $\cD_U$ is weakly approximable, and $\cD$ has a compact
    generator $G$ such that $\Hom {\cD}
    {\Sigma^{-n} G} {G} = 0$ for $n \gg 0$.  
    Then the intrinsic categories
    $(\cD_U)^-_c$, $(\cD_U)^b_c$, $\cD^-_c$, $\cD_c^b$, $(\cD_F)^-_c$
    and $(\cD_F)^b_c$,
    corresponding to 
    the preferred equivalence
    class of t-structures as in Remark~\ref{preferred-Ds},
    are all thick, triangulated subcategories of (respectively) $\cD_U$, $\cD$ and $\cD_F$.
\end{lemma}

\begin{proof}
For the categories
$(\cD_U)^-_c$, $(\cD_U)^b_c$, $\cD^-_c$ and $\cD_c^b$
the assertion is immediate from Remark~\ref{thickness}. What we will prove
is that the category $\cD_F$ has a compact generator $H$ with
$\Hom {\cD_F}{\Sigma^{-n} H} {H} = 0$ for $n \gg 0$.

    Choose for $\cD_F,\ \cD_U$ t-structures in the preferred equivalence class,
    and glue them to form a t-structure on $\cD$. By
    Corollary~\ref{importantcorollary} the glued t-structure on $\cD$ is 
    in the preferred equivalence class.
Pick a compact generator $G\in\cD$.

Lemma~\ref{lem:gen}(1) allows us to choose an integer $C>0$ with
$\Hom {\cD}{G} {\cD^{\leq-C}} = 0$. As $i_*$ is t-exact we have
$i_*\cD_F^{\leq-C}\subset\cD^{\leq-C}$, and hence
$$\Hom {\cD_F}{i^*G} {\cD_F^{\leq-C}}\cong\Hom {\cD}{G} {i_*\cD_F^{\leq-C}} = 0.$$
But the t-structure on $\cD_F$ is in the preferred equivalence class,
and \cite[Observation~0.20(ii)]{ApproxAmnon} informs us that the
compact object $i^*G\in\cD_F$ must be contained in $\cD_F^-$. Hence
we may choose an integer $B>0$ with $i^*G\in\cD_F^{\leq B}$.

But now it's immediate that the compact generator $i^*G\in\cD_F$ is such that
$\Hom {\cD_F}{\Sigma^{-n} i^*G} {i^*G} = 0$ for $n\geq B+C$.
\end{proof}

  \begin{prop}\label{prop:gluing-dcminus}
    Consider a recollement as in~Definition~\ref{def:rec}. Suppose
    $\cD_U$ is weakly approximable, and $\cD$ has a compact
    generator $G$ such that $\Hom {\cD}
    {\Sigma^{-n} G} {G} = 0$ for $n \gg 0$.  
    With $\cC$ standing for any of $\cD_U$, $\cD$ or $\cD_F$, let the categories
    $\cC^-_c$, $\cC_c^b$ and $\cC^+$ be the intrinsic ones coming from the preferred equivalence
    class of t-structures, see Remark~\ref{preferred-Ds}.
    If  $j^{*} G$ is in
    $(\cD_{U})^{-}_{c},$ then there  are equalities:
    \begin{align*}
      \Dcminus &= \left \{ X \in \cD \, | \, i^{*} X \in  (\cD_{F})^{-}_{c}
                 \text{ and } j^{*} X \in  (\cD_{U})^{-}_{c} \right \},\\
                 \Dcb \,
               &= \left \{ X \in \cD \, | \, i^{*} X \in
                 (\cD_{F})^{-}_{c},\  j^{*} X \in (\cD_{U})^{b}_{c} \text{ and }  i^{!} X \in
                 (\cD_{F})^{+}\right \}.
    \end{align*}
  \end{prop}

  \begin{proof}
  Choose t-structures for each of $\cD_F,\ \cD_U$, in the preferred equivalence
  classes,
    and glue them to form a t-structure on $\cD$. By
    Corollary~\ref{importantcorollary} the glued t-structure on $\cD$ is 
    in the preferred equivalence class.
    We are assuming that,
  for the compact generator $G\in\cD$,
  we have $j^{*} G \in (\cD_{U})^{-}_{c}$. As $ (\cD_{U})^{-}_{c}$ is
  a thick triangulated subcategory which contains $j^{*}G$, it must also contain 
  $j^*\cD^c$, since $\cD^c$ is the smallest thick triangulated subcategory of $\cD$ containing $G$.
 	
  We first show the left sides are contained in the right sides. Let $X$ be in
  $\Dcminus,$ fix $m > 0,$ and let $E \to X \to D$ be an exact triangle in $\cD$ with
  $E \in \cD^{c}$ and $D \in \Dl [-m].$ 
  Since $i^{*} D \in \Dl[-m]_{F}$ and $i^{*}$ preserves
  compactness (see~\ref{def:rec}), the exact triangle
  $i^{*} E \to i^{*} X \to i^{*} D$ shows that $i^{*} X$ is in
  $(\cD_{F})_{c}^{-}.$ Now consider the exact triangle $j^{*} E \to j^{*} X \to j^{*} D$
  with $j^{*} D \in \cD_{U}^{\leq -m}$. By the first paragraph of
the proof
  $j^*E \in (\cD_{U})^{-}_{c}$,
  hence we can find an exact triangle
  $\t E \to j^{*} E \to \t D$ with $\t E$ in $(\cD_{U})^{c}$ and $\t D$
  in $\cD_{U}^{\leq -m}.$ The octahedral axiom, applied to
  $\t E \to j^{*}E \to j^{*}X$, gives an object $D'$ and
  exact triangles $\t E \to j^{*} X \to D'$ and  $\t D \to D' \to j^{*} D$ in $\cD_U$ such
  that the following diagram is commutative:
  \begin{displaymath}
    \begin{tikzcd}
      \t E \arrow{d} \arrow[dotted]{rd}
      &&\\
      j^{*} E \arrow{d} \arrow{r} & j^{*} X \arrow[dotted]{rd}
      \arrow{r}
      & j^{*}D\\
      \t D \arrow[dotted]{rr} && D'.  \arrow[dotted]{u}
    \end{tikzcd}
  \end{displaymath}
  Since $\t D$ and $j^{*} D$ are in $\cD_{U}^{\leq -m},$ we know $D'$ is also in $\cD_{U}^{\leq -m}$. The exact triangle
  $\t E \to j^{*} X \to D'$ now shows that $j^{*} X$ is in
  $(\cD_{U})_{c}^{-}$. If we assume further that $X$ is in $\Dcb =
  \Dcminus \cap \cD^{b},$ then
  we have $j^{*} X \in \cD^{b}_{U}$ and $i^{!} X \in (\cD_{F})^{+}$ because $j^*$ is t-exact and $i^!$ is left t-exact.


  We now show the right sides are contained in the left sides. Assume $i^{*} X \in (\cD_{F})^{-}_{c}$ and
  $j^{*} X \in (\cD_{U})^{-}_{c}$, and fix $m > 0$. By definition there exists
  an exact triangle
  \[E' \to i^{*} X \to D'\]
  with $E' \in (\cD_{F})^{c}\subset\cD^-$ and $D ' \in \cD_{F}^{\leq -m}$.
  Choose an odd integer $n>0$ with $\Sigma^n E'\in\cD_F^{\leq-m}$. The object
  $E'\oplus \Sigma^nE'$ vanishes in $K_0(\cD_{F})$, and
  \cite[Corollary 4.5.14]{MR1812507} tells us that
  $E'\oplus \Sigma^nE'$ must therefore lie in the image of $i^*$.
  And then \cite[Proposition 4.4.1]{MR1812507}, applied to 
  the composite $E'\oplus\Sigma^nE'\to E' \to i^{*} X$,  
  allows us to find a morphism $E'' \to X$ in $\cD,$ with $E''$ in $\cD^{c},$
  and such
  that $i^{*}(E'' \to X)$ is isomorphic to $E'\oplus\Sigma^nE' \to E'\to i^{*} X.$ Complete $E'' \to X$ to an exact
  triangle 
  $$E'' \to X \to D''.$$
  Since $i^*D''\cong D'\oplus\Sigma^{n+1}E'$ we deduce that
  $i^*D''\in\cD_F^{\leq-m}$.
  
  The first paragraph of the proof tells us that
  $j^{*} E'' \in (\cD_{U})_{c}^{-}$,  while
$j^{*} X \in (\cD_{U})_{c}^{-}$ by hypothesis.
  Thus $j^{*} D''$ is in
  $(\cD_{U})_{c}^{-}$. Fix an exact triangle
  $\t E \to j^{*} D'' \to \t D$ with $\t E \in (\cD_{U})^{c}$ and
  $\t D \in \Dl[-m]_{U}.$ We now have the following diagram 
$$
\begin{tikzcd}
  & j_{!}\t E \arrow{r} &
  j_{!} j^{*} D'' \arrow{r}\arrow{d}& j_{!} \t D\\
  E'' \arrow{r}
  & X \ar[r] & D'' \ar[d]& \\
  && i_{*}i^{*} D'', &
\end{tikzcd}
$$
with $E''\in \cD^{c}, \ \t E \in (\cD_{U})^{c}$ 
and $\t D \in \cD_{U}^{\leq -m}$.
Applying the octahedral axiom, to $j_{!} \t E \to j_{!}j^{*} D'' \to D''$,
we find an object $D$ and exact triangles
$j_{!} \t E \to D''\to D$ and $j_{!}\t D \to D \to i_{*} i^{*} D''$ in $\cD$
that fit into the following commutative diagram:
$$
\begin{tikzcd}
  j_{!}\t E \ar[r] \ar[dr, dotted] &
  j_{!} j^{*} D'' \ar[r]\ar[d] & j_{!} \t D \ar[dd, dotted]\\
  & D'' \ar[d] \ar[dr, dotted] & \\
  & i_{*}i^{*} D'' & D.\ar[l, dotted]
\end{tikzcd}
$$
Since $i^{*} D''  \in \cD_{F}^{\leq -m}$ and $i_{*}$ is
t-exact, we know $i_{*}i^{*} D'' \in \Dl[-m]$. Since $\t D \in \Dl[-m]_{U}$ and
$j_{!}$ is right t-exact, we know $j_{!} \t D \in \Dl[-m].$ Thus $D \in \Dl [-m].$

Next we complete the octahedron on $X \to D'' \to D$, and find an object
$E$ with exact triangles $E \to X \to D$ and
$E'' \to E \to j_{!} \t E$ in $\cD$, fitting into the following commutative
diagram:
$$
\begin{tikzcd}
  E \ar[rr, dotted] \ar[dr, dotted]&& j_{!}\t E \ar[d] &\\
  E'' \ar[u, dotted] \ar[r] & X \ar[r] \ar[dr, dotted]& D'' \ar[d]\\
  && D.
\end{tikzcd}
$$
Note that $j_{!}\t E$ is compact because $j_{!}$ preserves compactness. Since $E''$ is also compact, so is $E$. 
Hence the exact triangle $E \to X \to D$ shows that $X$ is in $\Dcminus.$
If we now further assume that $i^{!} X \in (\cD_{F})^{+}$ and 
$j^{*} X \in (\cD_{U})^{b},$ then $X \in \cD^{b}$ by the definition of
glued t-structure.
\end{proof}

  \section{Gluing approximability}

  The time has come to prove the main theorem. For the reader's convenience
  we recall the statement.
  
\begin{theorem}\label{thm:gluing-result}
  Let $\cD_{F}$ and $\cD_{U}$ be approximable triangulated categories
  and let $\cD$ be a compactly generated triangulated category. Assume
  there is a recollement
  \begin{displaymath}
    \begin{tikzcd}
      \cD_{F} \arrow[r, "i_{*}"] & \cD \arrow[l,shift left = 4.3,
      "i^{!}"']  \arrow[l,shift right = 4.3, "i^{*}"'] \arrow[r,
      "j^{*}"] & \cD_{U}.  \arrow[l,shift left = 4.3, "j_{*}"']
      \arrow[l, shift right = 4.3, "j_{!}"']
    \end{tikzcd}
  \end{displaymath}
  Then $\cD$ has a compact generator $G$. If moreover
  $\Hom {\cD}{\Sigma^{-n} G} {G} = 0$ for~$n \gg 0,$ then $\cD$ is
  approximable.
\end{theorem}

We will prove the theorem via a series of lemmas. If $\cD_F$ is approximable with compact generator $G_F$, then there exists $A>0$ such that for every object $X$ of $\Dl$, we can find an exact triangle
$$E'\xra{f'} i^{*}X\to D'$$ in $\cD_F$ with $E'\in \o{\la G_F \ra}_{A}^{[-A,A]}$ and $D'\in \cD_F^{\leq -1}$.
To prove the theorem, we will show how to lift a morphism $E'\xra{f'} i^*X$ in $\cD_F$ to a morphism $E\xra{f} X$ in $\cD,$ for increasingly complicated~$E'$.

\begin{lem}\label{lem:gap}
   Consider a recollement that satisfies the hypotheses
   of~Corollary~\ref{importantcorollary}, and let $\dtstr$ be
   a t-structure in the preferred equivalence class on~$\cD.$ For every
   integer $A>0$ there exists an integer $B>0$ such that, if
  $f':i^{*}K\rightarrow i^{*}X$ is a morphism in $\cD_F$ with $K$
  compact and $K,X\in\Dl[A]$, then there exists an object $Y\in\Dl[B]$
  and morphisms
  \begin{displaymath}
    \begin{tikzcd}
      K& Y \ar[l, "\psi"'] \ar[r, "f"] &X,
    \end{tikzcd}
  \end{displaymath}
  such that $i^{*}(\psi)$ is an isomorphism in $\cD_F$ and
  $f' i^{*}(\psi)=i^{*}(f)$.
\end{lem}

\begin{proof}
  Choose t-structures for each of $\cD_F,\ \cD_U$, in the preferred equivalence
  classes,
    and glue them to form a t-structure on $\cD$. By
    Corollary~\ref{importantcorollary} the glued t-structure on $\cD$ is 
    in the preferred equivalence class, hence equivalent to the given
    t-structure on $\cD$. Since replacing the given t-structure by
    the equivalent glued one is harmless,
    let us assume that the
t-structure on $\cD$ is the one obtained from the gluing.

Now choose a compact generator $G\in\cD$ and a compact generator $H\in\cD_U$.
The object $G\oplus j_!H$ is a compact generator for $\cD$,
and Lemma~\ref{lem:gen} permits us to choose an integer
$C>0$ with $\Hom{\cD}{G\oplus j_!H}{\cD^{\leq-C}}=0$.
Assume also that $C\geq A$, where $A>0$ is the integer given
in the Lemma.
By the weak approximability of $\cD_{U}$ and~Lemma~\ref{rem:wa} we may,
possibly at the cost of increasing $C$, assume 
that every object $Z \in \Dl_{U}$ admits an exact triangle $E \to Z \to D$ with
  $E \in \o{\la H \ra}^{[-C,C]}$ and
$D \in \Dl[-1]_{U}$. We assert that in the Lemma we may set $B=3C-1$.

By \cite[Proposition 4.4.1]{MR1812507}, we can
  represent $f':i^{*}K\rightarrow i^{*}X$ in $\cD_F$ by a
  roof in $\cD$,
  \begin{displaymath}
    \begin{tikzcd}
      K& Y \ar[l, "\varphi"'] \ar[r, "\t f"] &X,
    \end{tikzcd}
  \end{displaymath}
  such that $Y$ is compact, $i^{*}(\varphi)$ is an isomorphism in $\cD_F$
  and $f' i^{*}(\varphi)=i^{*}(\t f)$. Since $i^{*}(\varphi)$ is an
  isomorphism, we can find an exact triangle $j_{!}Z\rightarrow Y\xrightarrow{\varphi} K$ for some $Z\in \cD_U$.  Moreover, we know $Y\in\cD^c\subset\cD^-$ and $K\in\Dl[A]$. Hence $j_{!}Z\in\Dl[N]$ for some
  $N$, and if $N\leq 3C-1$ we are done.
  Assume therefore $N\geq 3C$.  The isomorphism
  $Z\cong j^{*}j_{!}Z$ tells us that
  $Z\in\Dl[N]_U$. By \cite[2.2.1]{ApproxAmnon}
  (with $m =N+1-3C$ and $F = \Sigma^{N} Z$), there is an exact triangle
  $E\rightarrow Z\rightarrow D$ in $\cD_U$ with
  $E \in \o{\la H\ra}^{[2C,N+C]}$ and
  $D \in \Dl[3C-1]_{U}$. We have the diagram
  \begin{displaymath}
    \begin{tikzcd}
      j_{!}E\arrow{r} &j_{!}Z\arrow{d}\arrow{r}
      &j_{!}D\\
      & Y \ar[d, "\varphi"']\\
      &K
    \end{tikzcd}
  \end{displaymath}
which we complete to an octahedron,  giving an object $Y'$ and exact
  triangles $j_{!}E\rightarrow Y\rightarrow Y'$ and $j_{!}D\rightarrow
  Y'\xrightarrow{\psi} K$ in $\cD$, making the following diagram commutative:
  \begin{displaymath}
    \begin{tikzcd}
      j_{!}E\arrow{r}\arrow[dotted]{rd} &j_{!}Z\arrow{d}\arrow{r}
      &j_{!}D\arrow[dotted]{dd}\\
      & Y \ar[d, "\varphi"']  \ar[dr,dotted]&\\
      &K &Y'\arrow[dotted]{l}{\psi}.
    \end{tikzcd}
  \end{displaymath}
  Using that $\Hom {\cD} {j_!H} {\Dl [-C]} = 0$, we see
  that there are no nonzero maps from
  $j_{!}E\in \o{\la j_{!}H\ra}^{[2C,N+C]}$ to
  $X\in\Dl[A]\subset\Dl[C]$.
  Hence the morphism $\t f: Y \rightarrow X$ factors via
  $Y\rightarrow Y'$:
    \begin{displaymath}
    \begin{tikzcd}
      j_{!}E\arrow{r}\arrow[dotted]{rd} &j_{!}Z\arrow{d}\arrow{r}
      &j_{!}D\arrow[dotted]{dd}\\
      & Y \ar[d, "\varphi"']  \ar[dr,dotted] \ar[rr, bend left = 20, "\t
      f" near end] &&X\\
      &K &Y'\ar[l, dotted, "\psi"]
\ar[ur, bend right = 20, dotted, "f"']
    \end{tikzcd}
  \end{displaymath}
  We have thus found morphisms
  \begin{displaymath}
    \begin{tikzcd}
      K& Y' \ar[l, "\psi"'] \ar[r, "f "] &X,
    \end{tikzcd}
  \end{displaymath}
  such that $i^{*}(\psi)$ is an isomorphism in $\cD_F$ and
  $f' i^{*}(\psi)=i^{*}(f)$. Furthermore, the exact
  triangle $j_{!}D\rightarrow Y'\rightarrow K$ with
  $j_{!}D \in \Dl[3C-1]$ and $K\in \Dl[A]$ shows
  $Y'\in \Dl[3C-1]$. This proves the lemma.
\end{proof}

\begin{lem}\label{lem:case-of-n-equals-1}
  Consider a recollement that satisfies the hypotheses
  of~Corollary~\ref{importantcorollary}.
  Choose a compact generator $G\in\cD$ and let $\dtstr$ be
  a t-structure in the preferred equivalence class on $\cD$.
  
  For every integer $A>0$ there exists an integer $B>0$
  such that, if $f':E'\rightarrow i^{*}X$ is a morphism in $\cD_{F}$
  with $E'\in \Coprod_1 (i^*G[-A,A])$ and $X\in\Dl[A]$, then there
  exists a morphism $f:E\rightarrow X$ in $\cD$ with
  $E\in \o{\la G \ra}^{[-B,B]}$ such that $i^{*}E\cong E'$ and
  $i^{*}f\cong f'$.
  
  If $\cD_{U}$ is approximable, then the integer $B$ may be chosen so that
  $E\in \o{\la G \ra}_{B}^{[-B,B]}.$
\end{lem}

\begin{proof}
Corollary~\ref{importantcorollary} allows us to
  choose t-structures for $\cD_F,\ \cD_U$ in the preferred equivalence classes,
  glue them to form a t-structure on $\cD$, and replace the given
  t-structure on $\cD$ by the equivalent glued one we have just constructed.

Now choose a compact generator  $H\in\cD_U$.
The object $G$ is a compact generator for $\cD$,
and Lemma~\ref{lem:gen}(1) premits us to choose an integer
$C>0$ with $\Hom{\cD}{G}{\cD^{\leq-C}}=0$. Since
$G$ is a compact generator in $\cD$ and $j_!H$ is compact,
\cite[Lemma 0.9(i)]{ApproxAmnon} allows us to
increase $C$ so that $j_!H\in\o{\la G\ra}^{[-C,C]}_C$.
Assume also that $C\geq A$, where $A>0$ is the integer given
in the Lemma. Let $A'>0$ be an integer with $G\in\cD^{\leq A'}$;
Lemma~\ref{lem:gap} permit us to produce an integer $B'>0$ so that
any morphism $i^*K\to i^*X$, with $K\in\cD^c\cap\cD^{\leq A+A'}$
and $X\in\cD^{\leq A+A'}$, can be represented by a roof with $Y\in\cD^{\leq B'}$;
at the cost of possibly increasing $C$ assume $B'\leq C$. 
Finally the weak approximability (respectively approximability)
of $\cD_{U}$ and~Lemma~\ref{rem:wa} permits us,
possibly at the cost of increasing $C$, to assume 
that every object $Z \in \Dl_{U}$ admits an exact triangle $E \to Z \to D$ with
$D \in \Dl[-1]_{U}$ and   $E \in \o{\la H \ra}^{[-C,C]}$ (respectively
$E \in \o{\la H \ra}_C^{[-C,C]}$).
We assert that in the Lemma we may set $B=\max(4C,3C^3+1)$.

  It suffices to show the lemma for
  $E' = \Sigma^{n} i^*G$ with $n \in [-A, A]$. Consider a morphism
  $f':\Sigma^{n} i^*G\to i^{*} X$ in $\cD_F$ with
  $X\in\Dl[A]$. By the choice of $C$ there exists an object $Y$ in
  $\Dl[C]$ and morphisms
  \begin{displaymath}
    \begin{tikzcd}
      & Y \ar[ld, "\psi"'] \ar[rd, "\t f"] &\\
      \Sigma^{n} G & & X,
    \end{tikzcd}
  \end{displaymath}
  such that $i^{*}(\psi)$ is an isomorphism in $\cD_F$ and
  $f'i^{*}(\psi)=i^{*}(\t f)$. In particular, we can find an exact triangle
  $Y\xra{\psi} \Sigma^n G\to j_{!} Z$
for some $Z\in \cD_U$.
 Since $\Sigma^{n} G$ and $Y$ are both in $\Dl [C],$ so is $j_{!} Z$.
  Hence $Z\cong j^{*} j_{!} Z\in\Dl [C]_{U}$.
  By \cite[2.2.1 and 2.2.2]{ApproxAmnon}
  (with $m = 3C$ and $F=\Sigma^CZ$), there exists an exact triangle in
  $\cD_U$
  \begin{displaymath}
    \t E \to Z \to \t D,
  \end{displaymath}
  with
  $\t D \in \Dl [-2C]_{U}$ and $\t E\in \o{\la H \ra}^{[1-3C, 2C]}$, or
  if $\cD_U$ is approximable with
  $\t E\in \o{\la H \ra}^{[1-3C, 2C]}_{3C^2}$.
  We thus get the following diagram in $\cD$:
  \begin{displaymath}
    \begin{tikzcd}
      & Y \ar[d, "\psi"'] 
      \\
      &\Sigma^{n} G\ar[d] & 
      \\
      j_{!} \t E\ar[r] & j_{!} Z \ar[r] & j_{!} \t D .&{}
    \end{tikzcd}
  \end{displaymath}
  Since $j_{!}$ is
  right t-exact we have $j_{!} \t D \in \Dl [-2C]$. Since
  $\Hom {\cD} G {\Dl [-C]} = 0$ there
  are no nonzero maps $\Sigma^{n}G \to j_{!} \t D$. 
  It follows that $\Sigma^{n}G \to j_!Z$ factors via $j_{!} \t E \to j_!Z$.
  We thus find a morphism $\Sigma^{n} G \to j_{!} \t E$,
  an exact triangle
  $E\xra{\varphi} \Sigma^{n} G \to j_{!} \t E$ and a morphism
  $E\xra{g} Y$ such that the following diagram commutes:
  \begin{displaymath}
    \begin{tikzcd}
      & Y \ar[d, "\psi"'] 
      & E\ar[l, dotted, "g"] \ar[ld, dotted, "\varphi"]
      \\
      &\Sigma^{n} G\ar[dl, dotted] \ar[d]  
      \\
      j_{!} \t E\ar[r] & j_{!} Z \ar[r] & j_{!} \t D &{}
    \end{tikzcd}
  \end{displaymath}
  Since $j_{!}H\in\o{\la G\ra}^{[-C,C]}_C$ and
  $\t E$ is in $\o{\la H \ra}^{[1-3C, 2C]}$ (respectively
  $\t E\in\o{\la H \ra}^{[1-3C, 2C]}_{3C^2}$),
  we see that
  $j_{!}\t E\in \o{\la G\ra}^{[1-4C,3C]}$ (respectively
  $j_!\t E\in\o{\la G \ra}^{[1-4C, 3C]}_{3C^3}$).
  The triangle
	$$E\to \Sigma^n G\to j_{!}\t E$$ now shows that
  $E \in \o {\la G\ra}^{[-4C, 4C]}$ (respectively
  $E \in \o {\la G\ra}^{[-4C, 4C]}_{3C^3+1}$).
	Finally we let $f:=\t f g: E\to X$.
	Now $i^{*}(\varphi)$ is an isomorphism in $\cD_F$ and
      $$f'i^{*}(\varphi)=f'i^{*}(\psi)i^{*}(g)=i^{*}(\t fg)=i^{*}(f).$$
      \end{proof}

      \begin{lem}\label{lem:case-of-arb-n}
      	Consider a recollement that satisfies the hypotheses
        of~Corollary~\ref{importantcorollary} and suppose $\cD_U$ is
        approximable. Choose a compact generator $G\in\cD$ and
        let $\dtstr$ be
        a t-structure in the preferred equivalence class on $\cD$.

        For every pair of integers
        $A>0$ and $n>0$ there exists an integer $B>0$ such that, if
        $f':E'\rightarrow i^{*}X $ is a morphism in
  $\cD_{F}$ with $E'\in \Coprod_n (i^*G[-A,A])$ and $X\in\Dl[A]$, then
  there exists a morphism $f:E\rightarrow X$ in $\cD$ with
  $E\in \o{\la G \ra}_B^{[-B,B]}$ such that $i^{*}E\cong E'$ and
  $i^{*}f\cong f'$.
\end{lem}

\begin{proof}
Corollary~\ref{importantcorollary} allows us to
  choose t-structures for $\cD_F,\ \cD_U$ in the preferred equivalence classes,
  glue them to form a t-structure on $\cD$, and replace the given
  t-structure on $\cD$ by the equivalent glued one we have just constructed.

  The case $n=1$ is shown in~Lemma~\ref{lem:case-of-n-equals-1}. We
  proceed by induction on $n$, and suppose the lemma holds for
  some $n\geq 1$. Fix $A>0$, choose $\t B$ as
  in~Lemma~\ref{lem:case-of-n-equals-1}, meaning
  any morphism $f':E'\to i^*X$ in the category
  $\cD_F$, with $E'\in\Coprod_1 (i^*G[-A,A])$ and
  $X\in\Dl[A]$,
  is ismorphic to $i^*f$ for some some morphism
  $f:E\rightarrow X$ in $\cD$ with
  $E\in \o{\la G \ra}_{\t B}^{[-\t B,\t B]}$. Increasing $\t B$ if necessary
  assume $\t B\geq A$ and $G\in\cD^{\leq \t B}$.
  
  Using the induction hypothesis choose $B'$ such that,
  if $f':E'\rightarrow i^{*}X $ is a morphism in $\cD_{F}$ with
  $E'\in \Coprod_n (i^*G[-2\t B,2\t B])$ and $X\in\Dl[2\t B]$,
  then there exists a
  morphism $f:E\rightarrow X$ in $\cD$ with
  $E\in \o{\la G \ra}_{B'}^{[-B',B']}$ such that $i^{*}E\cong E'$ and
  $i^{*}f\cong f'$.

  Consider a morphism
  $f':E'\to i^{*}X$ with $E'\in \Coprod_{n+1} (i^*G[-A,A])$ and
  $X\in\Dl[A].$ By definition there exists an exact triangle $E_{1}' \to E' \to E'_{n} $ with
  $E'_{1} \in \Coprod_1 (i^*G[-A,A])$ and
  $E'_{n} \in \Coprod_n (i^*G[-A,A])$. We can complete the octahedron
  on $E_1'\rightarrow E'\rightarrow i^{*}X$ to find exact triangles
  $E'\to i^*X\to Z'$, $E'_1\to i^*X\to Y'$ and $E'_n\to Y'\to Z'$ such that the diagram
  \begin{displaymath}
    \begin{tikzcd}
      E_{1}' \arrow{r}\arrow[dotted]{rd} &E'\ar[d,"f'"]\arrow{r}
      &E_{n}'\arrow[dotted]{dd}\\
      & i^{*}X \ar[d]  \ar[dr,dotted]&\\
      &Z' &Y'\arrow[dotted]{l}
    \end{tikzcd}
  \end{displaymath}
  commutes.
  Since $E'_{1}$ is in $\Coprod_1 (i^*G[-A,A])$ and $X$ is in $\Dl[A]$,
we can apply our choices to find an exact
triangle $E_1\to X\to Y$ in $\cD$ with
  $E_1\in \o{\la G \ra}_{\t B}^{[-\t B,\t B]}$ and such that
  $i^{*}(E_1 \to X \to Y)\cong E_1'\to i^{*}X\to Y'$. Note that
 $$Y\in\Dl[A]*\o{\la G \ra}_{\t B}^{[-\t B-1,\t B-1]}\subseteq \Dl[2\t B].$$
  Next, we consider the morphism $E_n'\to Y'\cong i^{*}Y$ with
  $E'_{n} \in \Coprod_n (i^*G[-A,A])\subset \Coprod_n (i^*G[-2\t B,2\t B])$ and
  $Y\in\Dl[2\t B]$.  By our choices of integers
  there exists an exact triangle $E_n\to Y\to Z$ in
  $\cD$ with $E_n\in \o{\la G \ra}_{B'}^{[-B',B']}$ and such that
  $i^{*}(E_n\to Y\to Z)\cong E_n'\to Y'\to Z'$. We now have the
  following diagram:
    \begin{displaymath}
    \begin{tikzcd}
      E_{1} \arrow{rd} &
      &E_{n}\arrow{dd}\\
      & X \ar[dr]  &\\
      &Z &Y\arrow{l}.
    \end{tikzcd}
  \end{displaymath}
 Completing the octahedron, we find an object $E$ and exact triangles
 $E_{1} \to E \to E_{n}$ and $E \to X \to Z$ in $\cD$,
 so that the following diagram
 is commutative:
     \begin{displaymath}
    \begin{tikzcd}
      E_{1} \arrow{rd} \ar[r, dotted]& E \ar[d, dotted, "f"] \ar[r, dotted]
      &E_{n}\arrow{dd}\\
      & X \ar[dr] \ar[d, dotted] &\\
      &Z &Y\arrow{l}.
    \end{tikzcd}
  \end{displaymath}
Since there are isomorphisms $i^{*}(X\to Y)\cong i^{*}X\to Y'$ and
  $i^{*}(Y\to Z)\cong Y'\to Z'$, there is an isomorphism $i^{*}f\cong f'$.  The lemma follows since
 $$E \in \o{\la G \ra}_{\t B}^{[-\t B,\t B]}*\o{\la G \ra}_{B'}^{[-B',B']}
  \subset \o{\la G \ra}_{\t B+B'}^{[-\t B-B',\t B+B']}.$$
\end{proof}

      \begin{lem}\label{lem:generalcase}
	    	Consider a recollement that satisfies the hypotheses
        of~Corollary~\ref{importantcorollary} and suppose $\cD_U$ is
        approximable. Choose a compact generator $G\in\cD$ and
        let $\dtstr$ be
        a t-structure in the preferred equivalence class on $\cD$.

        For any integer
        $A>0$ there exists an integer $B>0$ such that, if
        $f':E'\rightarrow i^{*}X $ is a morphism in
        $\cD_{F}$ with $E'  \in \o{\la i^*G\ra}_A^{[-A,A]}$
        and $X\in\Dl[A]$, then
  there exists a morphism $f:E\rightarrow X$ in $\cD$ with
  $E\in \o{\la G \ra}_B^{[-B,B]}$ and such that $i^{*}E\cong E'$ and
  $i^{*}f\cong f'$.
\end{lem}

\begin{proof}		
  Fix $A>0$.  In \cite[Lemma 1.11]{AmnonStrong} it is shown that
	$$\o{\la i^*G \ra}_A^{[-A,A]}\subset \Coprod_{2A}
  (i^*G[-A-1,A])\subset \Coprod_{2A}
  (i^*G[-2A,2A]).$$
  By~Lemma~\ref{lem:case-of-arb-n} there
        exists $B>0$ such that if
        $f':E'\rightarrow i^{*}X $ is a morphism in $\cD_{F}$ with
        $E'\in \Coprod_{2A} (G_F[-2A,2A])$ and $X\in\Dl[2A]$, then there exists
        $E\in \o{\la G \ra}_B^{[-B,B]}$ and a morphism
        $f:E\rightarrow X$ in $\cD$ with $i^{*}E\cong E'$ and $i^{*}f\cong f'$.
      \end{proof}
  
 Now that we understand how to lift morphisms in $\cD_F$ to $\cD$, we proceed with the proof of the theorem:

        \begin{proof}[Proof of Theorem \ref{thm:gluing-result}.]
          Let $G_{F}$ and $G_{U}$ be compact generators of $\cD_F$ and
          $\cD_U$, respectively. By \cite[Theorem 4.4.9]{MR1812507}
          there exists a
          compact object $G'$ in $\cD$ such that $G_{F}$ is a direct
          summand of~$i^{*} G'$. It follows that
          $G:= G' \oplus j_{!}  G_{U}$ is a compact generator
          for~$\cD$.
        
          Suppose now that $\Hom {\cD} {\Sigma^{n} G} G = 0$ for
          $n \gg 0$; we're in the situation of
          Corollary~\ref{importantcorollary}, 
 allowing us to
  choose t-structures for $\cD_F,\ \cD_U$ in the preferred equivalence classes,
  glue them to form a t-structure on $\cD$, and we're guaranteed that
  the glued t-structure on $\cD$ is in the preferred equivalence class.
  Choose a compact generator $H\in\cD_U$, and choose an integer $A>0$
  so that $G\in\cD^{\leq A}$, $\Hom{\cD}G{\cD^{\leq-A}}=0$ and
  $j_!H\in{\la G \ra}_{A}^{[-A,A]}$.
 Recalling that $\cD_U$ and $\cD_F$ are
  approximable and remembering~Remark~\ref{rem:wa}
  we may,  by increasing $A$ if necessary, also assume 
   that every object
          $X\in \Dl_U$ admits an exact triangle $E\to X\to D$ with
          $E\in\o{\la H \ra}_{A}^{[-A,A]}$ and 
          $D\in \Dl[-1]_U$,
 and  every object
          $X\in \Dl_F$ admits an exact triangle $E\to X\to D$ with
          $E\in\o{\la i^*G \ra}_{A}^{[-A,A]}$ and
          $D\in \Dl[-1]_F$.  
          By~Lemma~\ref{lem:generalcase} there
          exists an integer $B>0$ such that, if
          $f':E'\rightarrow i^{*}X $ is any morphism in $\cD_{F}$ with
          $E'\in \o{\la i^*G \ra}_{A}^{[-A,A]}$ and $X\in\Dl[A]$,
          then there exists a
          morphism $f: E \to X$ in $\cD$ with
          $E\in \o{\la G \ra}_{B}^{[-B,B]}$ such that $i^{*} f \cong f'.$

          Let $X$ be an object of $\Dl$, hence
          $i^{*} X \in \Dl_{F}.$ We find an exact triangle
$$E'\xra{f'} i^{*}X\to D'$$ with $E'\in \o{\la i^*G \ra}_{A}^{[-A,A]}$ and $D'\in \cD_F^{\leq -1}$.
By the choices of integers above there exists
a morphism $E \xra{f} X $ in $\cD$ with
$E\in \o{\la G \ra}_{B}^{[-B,B]}$ such that
$i^{*}f\cong f' $. Complete $f$ to a triangle $$E \xra{f} X \to D''.$$
Since $E \in \o{\la G \ra}_{B}^{[-B,B]}$ and $X\in \Dl,$ we have
$D'' \in \Dl[A+B],$ and thus $j^{*} D'' \in \Dl[A+B]_{\cU}.$
By~\cite[2.2.1]{ApproxAmnon} (with $F=\Sigma^{A+B}j^*D''$ and $m=A+B+1$), we find an exact triangle
$$\t E\to j^{*} D''\to \t D$$ in $\cD_U$ such that
$\t E\in \o{\la H \ra}_{(A+B+1)A}^{[-A,2A+B]}$ and
$\t D\in \Dl[-1]_U$, obtaining a diagram
$$
\begin{tikzcd}
  & j_{!}\t E \arrow{r} &
  j_{!} j^{*} D'' \arrow{r}\arrow{d}& j_{!} \t D\\
  E \arrow[r,"f"]
  & X \ar[r] & D'' \ar[d]& \\
  && i_{*}i^{*} D''. &
\end{tikzcd}
$$
Completing the octahedron on $j_{!} \t E \to j_{!}j^{*} D'' \to D''$,
we find an object $D$ with exact triangles
$j_{!} \t E \to D''\to D$ and $j_{!}\t D \to D \to i_{*} i^{*} D''$ in $\cD$,
that fit into the following commutative diagram:
$$
\begin{tikzcd}
  j_{!}\t E \ar[r] \ar[dr, dotted] &
  j_{!} j^{*} D'' \ar[r]\ar[d] & j_{!} \t D \ar[dd, dotted]\\
  & D'' \ar[d] \ar[dr, dotted] & \\
  & i_{*}i^{*} D'' & D.\ar[l, dotted]
\end{tikzcd}
$$
We note that $j_{!} \t D \in \Dl[-1]$ since $\t D \in \Dl[-1]_{U}$
and $j_{!}$ is right t-exact. Furthermore, $i_{*}i^{*} D'' \in \Dl [-1],$ since
$i^{*} D'' \cong D' \in \Dl[-1]_{F}$ and $i_{*}$ is t-exact. Thus
$D \in \Dl [-1].$

Next we complete the octahedron on $X \to D'' \to D$, finding an object
$F$ with exact triangles $F \to X \to D$ and
$E \to F \to j_{!} \t E$ in $\cD$, that fit into the following commutative
diagram:
$$
\begin{tikzcd}
  F \ar[rr, dotted] \ar[dr, dotted]&& j_{!}\t E \ar[d] &\\
  E \ar[u, dotted] \ar[r] & X \ar[r] \ar[dr, dotted]& D'' \ar[d]\\
  && D.
\end{tikzcd}
$$
Since $E\in \o{\la G \ra}_{B}^{[-B,B]}$ and
$$j_{!}\t E\in \o{\la j_{!}H \ra}_{(A+B+1)A}^{[-A,2A+B]}\subset
\o{\la G \ra}_{(A+B+1)A^2}^{[-2A,3A+B]},$$
we see that
$F \in \o{\la G \ra}_{(A+B+1)A^2+B}^{[-3A-B,3A+B]}$. Thus
the exact triangle $F \to X \to D$
proves the approximability of $\cD$.
\end{proof}

\section{Gluing dg-categories}

In this section we fix a commutative ring $R$, and assume everything
in sight is $R$-linear. Recall that an \emph{algebraic triangulated
  category} is a triangulated category that is equivalent to the
stable category of a Frobenius exact category. By \cite[Theorem
4.3]{MR1258406} if $\cD$ is a cocomplete algebraic triangulated
category with a single compact generator, then $\cD$ is equivalent to
the derived category of a dg-algebra. Thus, determining which
algebraic triangulated categories are approximable is equivalent to
determining which dg-algebras have an approximable derived category.

If $A$ is a dg-algebra and its derived category $D(A)$ is approximable, then
$H^{n}(A) = \Hom {D(A)} {\Sigma^{-n} A} A = 0$ for $n \gg 0$ by Lemma
\ref{lem:gen}. By \cite[Remark 3.3]{ApproxAmnon},
if $H^{n}(A) = 0$ for all $n > 0,$ then $D(A)$ is
approximable. We start this
section by showing that if $A$ is an ``upper triangular algebra''
constructed from
dg-algebras $B$ and $C,$ and $B, C$ have approximable derived
categories, then so does $A$. This gives examples of
dg-algebras with approximable derived categories and cohomology in
arbitrarily high degree.

\begin{cor}
  Let $B$ and $C$
  be dg-algebras, $M$ a $B \otimes C^{\op}$-module, and $A$ the
  dg-matrix algebra of this data, i.e.,
\begin{displaymath}
A =  \left[
  \begin{array}{rr}
    B & M\\
    0 & C
  \end{array}
\right], \quad \quad d_{A} = \left[
  \begin{array}{rr}
    d_{B} & d_{M}\\
    0 & d_{C}
  \end{array}
\right].
\end{displaymath}
Assume that $M$ is cohomologically bounded above, i.e., $H^{n}(M) = 0$
for $n \gg 0,$ and that $M$ has a
  semi-projective $B$-resolution that is also a $B\otimes C^{\op}$-module. Then
  if $D(B)$ and
  $D(C)$ are approximable, so is $D(A)$. Moreover,
  $D(A)$ satisfies the hypothesis of Proposition \ref{prop:gluing-dcminus}, so
  $D(A)^{-}_{c}$ and $D(A)^{b}_{c}$ are glued from their analogues
  over $B$ and $C$.
\end{cor}

\begin{rem}\label{rem:bimod-gluing-over-a-field}
  The condition on the resolution of $M$ is satisfied if $R$ is a field, or more
  generally, if $M$ and $B$ are flat over $R$ as graded
  modules. Indeed, one
  can then construct a semi-projective $B$-resolution of $M$
  using the bar construction, and this resolution retains the right
  $C$-action from~$M$.
\end{rem}

\begin{proof}
 By \cite[\S 3]{MR2821717} (see also \cite{MR2228328}), there is a recollement:
  \begin{displaymath}
    \begin{tikzcd}
      D(B) \arrow[r, "i_{*}"] 
      & D(A) \arrow[l,shift left = 4.3, "i^{!}"']  \arrow[l,shift
      right = 4.3, "i^{*}"'] \arrow[r, "j^{*}"] & D(C). \arrow[l,shift
      left = 4.3, "j_{*}"'] \arrow[l, shift right = 4.3, "j_{!}"']
    \end{tikzcd}
  \end{displaymath}
Note that $B$ and $C$ are
cohomologically bounded above, since their derived categories are
approximable, and therefore $A$ is also cohomologically bounded
above. Thus we may apply Theorem \ref{thm:gluing-result}, with $A =
G,$ to see $D(A)$ is approximable. By \cite[Lemma 3.11]{MR2821717},
$j^{*}(A) \cong C,$ and thus is in $D(C)^{-}_{c},$ so we may apply~Proposition~\ref{prop:gluing-dcminus}.
\end{proof}

\def\cprime{$'$}


\begin{thebibliography}{10}

\bibitem{MR1974001}
Leovigildo Alonso~Tarr{\'\i o}, Ana Jerem{\'\i as}~L\'opez, and Mar\'\i
  a~Jos\'e Souto~Salorio.
\newblock Construction of {$t$}-structures and equivalences of derived
  categories.
\newblock {\em Trans. Amer. Math. Soc.}, 355(6):2523--2543, 2003.

\bibitem{MR751966}
A.~A. Beilinson, J.~Bernstein, and P.~Deligne.
\newblock Faisceaux pervers.
\newblock In {\em Analysis and topology on singular spaces, {I} ({L}uminy,
  1981)}, volume 100 of {\em Ast\'erisque}, pages 5--171. Soc. Math. France,
  Paris, 1982.

\bibitem{MR1996800}
A.~Bondal and M.~van~den Bergh.
\newblock Generators and representability of functors in commutative and
  noncommutative geometry.
\newblock {\em Mosc. Math. J.}, 3(1):1--36, 258, 2003.

\bibitem{efimov1}
Alexander~I. Efimov.
\newblock Categorical smooth compactifications and generalized {H}odge-to-de
  {R}ham degeneration.
\newblock arXiv:1805.09283v1.

\bibitem{efimov}
Alexander~I. Efimov.
\newblock Homotopy finiteness of some dg categories from algebraic geometry.
\newblock ar{X}iv:1308.0135v2.

\bibitem{MR2228328}
Peter Jorgensen.
\newblock Recollement for differential graded algebras.
\newblock {\em J. Algebra}, 299(2):589--601, 2006.

\bibitem{MR1258406}
Bernhard Keller.
\newblock Deriving {DG} categories.
\newblock {\em Ann. Sci. \'Ecole Norm. Sup. (4)}, 27(1):63--102, 1994.

\bibitem{MR3728631}
Alexander Kuznetsov.
\newblock Semiorthogonal decompositions in algebraic geometry.
\newblock In {\em Proceedings of the {I}nternational {C}ongress of
  {M}athematicians---{S}eoul 2014. {V}ol. {II}}, pages 635--660. Kyung Moon Sa,
  Seoul, 2014.

\bibitem{MR3439086}
Alexander Kuznetsov and Valery~A. Lunts.
\newblock Categorical resolutions of irrational singularities.
\newblock {\em Int. Math. Res. Not. IMRN}, (13):4536--4625, 2015.

\bibitem{MR2609187}
Valery~A. Lunts.
\newblock Categorical resolution of singularities.
\newblock {\em J. Algebra}, 323(10):2977--3003, 2010.

\bibitem{MR2981713}
Valery~A. Lunts.
\newblock Categorical resolutions, poset schemes, and {D}u {B}ois
  singularities.
\newblock {\em Int. Math. Res. Not. IMRN}, (19):4372--4420, 2012.

\bibitem{MR2821717}
Daniel Maycock.
\newblock Derived equivalences of upper triangular differential graded
  algebras.
\newblock {\em Comm. Algebra}, 39(7):2367--2387, 2011.

\bibitem{AmnonStrong}
A.~{Neeman}.
\newblock {Strong generators in $D^{perf}(X)$ and $D^b_{coh}(X)$}.
\newblock arXiv:1703.04484.

\bibitem{ApproxAmnon}
Amnon Neeman.
\newblock Triangulated categories with a single compact generator and a brown
  representability theorem.
\newblock ar{X}iv:1804.02240v1.

\bibitem{Ne96}
Amnon Neeman.
\newblock The {G}rothendieck duality theorem via {B}ousfield's techniques and
  {B}rown representability.
\newblock {\em J. Amer. Math. Soc.}, 9(1):205--236, 1996.

\bibitem{MR1812507}
Amnon Neeman.
\newblock {\em Triangulated categories}, volume 148 of {\em Annals of
  Mathematics Studies}.
\newblock Princeton University Press, Princeton, NJ, 2001.

\bibitem{MR3535370}
D.~O. Orlov.
\newblock Gluing of categories and {K}rull-{S}chmidt partners.
\newblock {\em Uspekhi Mat. Nauk}, 71(3(429)):203--204, 2016.

\bibitem{MR3545926}
Dmitri Orlov.
\newblock Smooth and proper noncommutative schemes and gluing of {DG}
  categories.
\newblock {\em Adv. Math.}, 302:59--105, 2016.

\bibitem{MR2434186}
Rapha\"el Rouquier.
\newblock Dimensions of triangulated categories.
\newblock {\em J. K-Theory}, 1(2):193--256, 2008.

\bibitem{MR2451292}
Gon\c{c}alo Tabuada.
\newblock Higher {$K$}-theory via universal invariants.
\newblock {\em Duke Math. J.}, 145(1):121--206, 2008.

\bibitem{MR2077594}
Michel van~den Bergh.
\newblock Non-commutative crepant resolutions.
\newblock In {\em The legacy of {N}iels {H}enrik {A}bel}, pages 749--770.
  Springer, Berlin, 2004.

\end{thebibliography}
\end{document}